\documentclass[a4paper,twoside,11pt]{article}
\usepackage{amssymb}
\usepackage{amsmath}
\usepackage{amsthm}
\usepackage{latexsym}
\usepackage{amsfonts}
\usepackage{graphicx}
\usepackage{graphics}

\newcommand{\dis}{\displaystyle}
\theoremstyle{plain}
\newtheorem{thm}{Theorem}[section]   
\newtheorem{prop}[thm]{Proposition}
\newtheorem{lem}[thm]{Lemma}

\newtheorem{Def}[thm]{Definition}

\theoremstyle{definition}
\newtheorem{rem}[thm]{Remark}

\newtheorem{clm}[thm]{Claim}

\newcommand{\el}{\ell}

\newcommand{\ra}{\;\rightarrow\;}

\newcommand{\bi}{\beta}

\newcommand{\de}{\delta }

\newcommand{\e}{\varepsilon }
\newcommand{\f}{\varphi}

\newcommand{\vPsi}{\varPsi}

\newcommand{\la}{\lambda }

\newcommand{\si}{\sigma }

\newcommand{\ti}{\tau }

\newcommand{\oo}{\omega}
\newcommand{\C}{\mathbb{C}}

\newcommand{\R}{\mathbb{R}}

\newcommand{\N}{\mathbb{N}}

\newcommand{\bP}{\mathbb{P}}
\newcommand{\ssum}{\sum\limits}

\newcommand{\tg}{\widetilde{g}}
\newcommand{\tn}{\widetilde{n}}

\newcommand{\tS}{\widetilde{S}}

\newcommand{\tT}{\widetilde{T}}

\newcommand{\tq}{\tilde{q}}

\newcommand{\ld}{\ldots}

\newcommand{\sm}{\smallsetminus}

\newcommand{\qb}{$\quad\blacksquare$}

  \newtheorem{theorem}{Theorem}
\newtheorem{lemma}[theorem]{Lemma}

 \newcommand{\lan}{\langle}                       
  \newcommand{\ran}{\rangle}

\begin{document}
\pagestyle{myheadings}
\markboth{Generalized Harmonic Functions on Trees: Universality and Frequent Universality}{N. Biehler, E. Nestoridi and V. Nestoridis}
\title{\bf Generalized Harmonic Functions on Trees: Universality and Frequent Universality}
%
%
\author{N. Biehler, E. Nestoridi and V. Nestoridis}
\date{}
\maketitle
\begin{abstract}
Recently, harmonic functions and frequently universal harmonic functions on a tree $T$ have been studied, taking values on a separable Fr\'{e}chet space $E$ over the field $\C$ or $\R$. In the present paper, we allow the functions to take values in a vector space $E$ over a rather general field $\mathbb{F}$. The metric of the separable topological vector space $E$ is translation invariant and instead of harmonic functions we can also study more general functions defined by linear combinations with coefficients in $\mathbb{F}$. Unlike the past literature, we don't assume that $E$ is complete and therefore we present a new argument, avoiding Baire's theorem.
\end{abstract}
{\em AMS classification numbers}: 05C05, 60J50, 46M99, 30K99.\smallskip
\section{Introduction}\label{sec1}
\noindent
 In 1945, Menchoff \cite{13} proved the existence of a plethora of
trigonometric series $\break \sum_{n=-\infty}^{+\infty}
a_ne^{int}$, which satisfy  that every
complex measurable $2 \pi$-periodic function is the almost everywhere limit
of a subsequence of the partial sums. Moreover
the sequence $(a_j)$ can be chosen in $c_0(\mathbb{Z})$. This kind of universality,
as well as many others are proven to be generic with the aid of Baire's
Category Theorem \cite{6,9,10}. 
Recently, in analogy to Menchoff's
universality, the generic existence of universal
martingales on trees was established. The study of harmonic functions on trees and famous properties, such as universality, frequent universality and algebraic genericity has been a recent  topic of interest in analysis \cite{1,2,3,7}.  

As in these references, let $T$ be the set of vertices of a rooted tree with root $x_0$ and let $T_n$ be the set of vertices at $n$--th level of the tree, so that $T=\dis\bigcup^\infty_{n=0}T_n$, with $T_0=\{x_0\}$. We assume that all $T_n$ are finite. Also, for each $x\in T_n$ we assume that  the set $S(x) \subset T_{n+1} $ of children of $x$, contains at least two elements. Notice that the family $(S(x))_{x\in T_n}$ is a partition of $T_{n+1}$. 

Let $\mathbb{F}$ be a field, which has a separable topology. Let $E$ be a (non-necessarily complete) topological vector space over $\mathbb{F}$, that is separable metrizable with a translation invariant metric. 

A function $f:T\ra E$ is called generalized harmonic, if for every $x\in T$ we have that
\begin{eqnarray}
f(x)=\sum_{y\in S(x)}w(x,y)f(y)  ,\label{eq1}
\end{eqnarray}
where $w(x,y) \in \mathbb{F}\setminus \{ 0\}$ for $y \in S(x) $ are given elements of $\mathbb{F}$.

 We denote the set of generalized harmonic functions by $H(T,E)$.  To define universality, we assign probability weights $q(x,y)$, to every $x, y \in T$, such that $q(x,y)>0$ if and only if  $y\in S(x)$. The probability weights satisfy $\ssum_{y\in S(x)}q(x,y)=1$ for all $x\in T$, which in other words says that the matrix $Q=(q(x,y))_{x ,y\in T}$ is the transition matrix of a non-backtracking random walk on $T$.
 
 
 We define the boundary $\partial T$ of the tree $T$ to be the set of all infinite geodesics starting at $x_0$; that is, all $e=\{z_n:n=0,1,2,\ld\}$, such that $z_0=x_0$ and $z_{n+1}\in S(z_n)$ for every $n=0,1,2,\ld$\,. A vertex $x\in T$ defines a boundary sector $B_x\subset\partial T$, which is the set of all $e\in\partial T$ such that $x\in e$. For $x \in T_n, $ we define $p_n(B_x)$ to be the product $\dis\prod^{n-1}_{j=0}q(y_j,y_{j+1})$, where $y_0=x_0, y_n=x$ and $y_{j+1} \in S(T_{j})$ for all $j \in \{ 0,1, \ldots, n-1\}$ are the vertices of the unique path from the root to $x$. Therefore, for each $n$, $p_n$ can be viewed as a probability measure on $\{B_x, x \in T_n\}$.

The family $\{B_x\}_{x\in T_n}$ generates a $\si$-algebra $M_n$ on $\partial T$ and the above probabilities $p_n(B_x)$ extend to a probability measure $\bP_n$ on $M_n$. One can check that $M_n\subset M_{n+1}$ and $\bP_{n+1}|_{M_n}=\bP_n$ for every $n=0,1,2,\ld$\;. The Kolmogorov Consistency Theorem \cite{11} implies that there exists a unique probability measure $\bP$ on the $\si$-algebra $M$ of subsets of $\partial T$ generated by the union $\dis\bigcup^\infty_{n=0}M_n$, such that $\bP|_{M_n}=\bP_n$ for every $n \in\mathbb{N} $.

\begin{Def}\label{def2.1}
If $f\in E^T$ and $n\in\{0,1,2,\ld\}$ we define $\oo_n(f): \partial T \rightarrow E$ to be a function defined as follows. Let $e\in\partial T$ and $z$ the unique point belonging to $e\cap T_n$; then we set $\oo_n(f)(e)=f(z)$. Obviously, $\oo_n(f)$ is $M_n$-measurable and hence $M$-measurable.
\end{Def}

A generalized harmonic function $f$ is said to be universal if the sequence $\oo_n(f)$, $n=0,1,2,\ld$ is dense in the space $L^0(\partial T,E)$ of measurable functions $h:\partial T\ra E$ endowed with the topology of convergence in probability. Our first main result concerns the set $U(T,E)$ of universal, generalized harmonic functions. 
\begin{theorem}\label{one}
 The set $U(T,E)$ is a $G_\de$-dense subset of the space $H(T,E)$. 
\end{theorem}
Since we don't assume completeness, we prove Theorem \ref{one} without the use of Baire's theorem, which is the traditional tool to proving such a result. 
The next result concerns another property of $U(T,E)$ called algebraic genericity.
 \begin{theorem}[algebraic genericity]\label{two}
The set $U(T,E)$ contains a vector space, except 0, dense in $H(T,E)$.
\end{theorem}
 A function $f$ in $H(T,E)$ is called frequently universal, if for every non-void open set $V$ of $L^0(\partial T,E)$, the set $\{n\in\N:\oo_n(f)\in V\}$ has strictly positive lower density. The last result of this paper concerns the set $U_{FM}(T,E)$ of frequently universal functions.
 \begin{theorem}\label{three}
The set $U_{FM}(T,E)$ is a dense subset of $H(T,E)$. Under the additional assumption that $\ssum_{y\in S(x)}w(x,y)=1$ for all $x\in T$, $U_{FM}(T,E)\cup\{0\}$ contains a vector space dense in $H(T,E)$ and that $U_{FM}(T,E)$ is disjoint from a $G_\de$ and dense subset of $H(T,E)$.
 \end{theorem}

If the space $E$ is complete, then, Theorem \ref{three} says that $U_{FM}(T,E)$ is dense and meager in $H(T,E)$.
 
Results of similar flavor were proven in \cite{7}, where $E$ is a separable Fr\'{e}chet space over the field $\C$ or $\R$ and $w(x,y)=q(x,y)$. In this case, the space $H(T,E)$ is a closed subset of the complete metric space $E^T$; thus it is itself a complete metric space.
As a result, the proof of Theorem \ref{one} in this case is based on Baire's Category Theorem. Also, the set $U_{FM}(T,E)$ of frequently universal functions is proven to be dense and meager in $H(T,E)$. Furthermore the sets $U(T,E)$ and $U_{FM}(T,E)$ contain a vector space, except zero, which is dense in $H(T,E)$; that is, we have algebraic genericity.

Initially, we had a study of potential theory on trees \cite{12}, where the definition of harmonic functions was using convex combinations of the values of the function at the children of a node. Here, we deal with generalized harmonic functions on trees, whose definition is using linear combinations of the values of the function at the children of a node.
Articles related to our investigations, except \cite{2,3,7} are \cite{4,5,6,8,9,10,12,13,14}. The results of the present paper contain extensions of results proven in \cite{2,3,7}.

The starting point of the present paper was when we wanted to study the simple model of universal ``harmonic'' functions taking values 0 or 1; that is, in $\mathbb{F}_2$. In the case of a binary tree, let $x_1$ and $x_2$ be the children of a node $x$. In this case, we cannot have that $w(x,x_1) + w(x,x_2)=1$, they belong to the
field $\mathbb{F}_2$. This led to the two of the three new elements in this paper, the third being  the lack of completeness which obliges us to avoid the use of Baire's Theorem.

\section{Preliminaries}\label{sec2}
\noindent

We consider a field $\mathbb{F}$ endowed with a Hausdorff topology, such that addition and the inverse function are continuous. Let also $E$ be a separable topological vector space over the field $\mathbb{F}$, whose topology is induced by a translation invariant metric $d$. We do not assume that $(E,d)$ is a complete metric space.

A function $h:\partial T\ra E$ is $M$--measurable (respectively $M_n$--measurable), if, for every open set $V\subset E$ we have $h^{-1}(V)\in M$ (respectively $M_n$). We identify two $M$--measurable functions on $\partial T$, if they are equal almost everywhere with respect to the probability measure $\bP$. Thus, we obtain the space $L^0(\partial T,E)$ of $M$--measurable measurable functions. Convergence in probability in this space is defined using the metric
\[
d_{\bP}(h,g)=\int_{\partial T}\frac{d(h(e),g(e))}{1+d(h(e),g(e))}d\bP(e),
\]
where $h,g\in L^0(\partial T,E)$ and $d$ is the metric on $E$.

Since, $E$ is separable and each $M_n$ is finite, the space of $M_n$--measurable functions defined on $\partial T$ with values in $E$ is separable as well. It follows that $L^0(\partial T,E)$ is also separable and, hence, it has a dense sequence $h_n\in L^0(\partial T,E)$, $n=0,1,2,\ld$, where each $h_n$ is $M_{\ti(n)}$--measurable for some $\ti(n)\in\{0,1,2,\ld\}$.
\begin{lemma}
If $E\neq\{0\}$, the space $L^0(\partial T,E)$ does not contain isolated points, independently of the fact that $\mathbb{F}$ may contain isolated points, as for example in the case $\mathbb{F}=\mathbb{F}_2$.
\end{lemma}
\begin{proof}
Let $h\in L^0(\partial T,E)$ and $\e>0$. Since for every $x\in T$ we have $\vert S(x) \vert \ge2$ and $q(x,y)>0$ for every $y\in S(x)$, there exists $y\in S(x)$ with $0<q(x,y)\le\dfrac{1}{2}$. It follows that there exists $n\in\{0,1,2,\ld\}$ and $x\in T_n\subset T$, such that $0<p(B_x)\le\dfrac{1}{2^n}<\e$. We consider the $M$--measurable function $g:\partial T\ra E$ defined as follows: $g(e)=h(e)$ if $e\in\partial T\sm B_x$ and for $e\in B_x$ $g(e)=0$ if $h(e)\neq0$ and $g(e)=v$ if $h(e)=0$, where $v\in E-\{0\}$ is arbitrary but fixed. One easily checks that $d_\bP(h,g)<\e$ and that the functions $h$ and $g$ are not $\bP$-almost everywhere equal, because they are different on $B_x$ and $\bP(B_x)>0$. Thus, $L^0(\partial T,E)$ does not contain isolated points.
\end{proof}
For $x\in T$ and $y\in S(x)$ we fix elements $w(x,y)\in  \mathbb{F} \setminus \{0\}$.
We consider the space $E^T=\{f:T\ra E\}$ endowed with the metric
\[
\rho(f,g)=\sum^\infty_{n=0}\frac{1}{2^n}\cdot\frac{d(f(x_n),g(x_n))}{1+d(f(x_n),g(x_n))},
\]
where $x_n$, $n=0,1,2,\ld$ is a fixed enumeration of $T$. The metric $\rho$ induces the cartesian topology on $E^T$.

In this paper, we are interested in functions $f\in E^T$, such that the sequence $\oo_n(f)$, $n=0,1,2,\ld$ is dense in $L^0(\partial T,E)$ or such that, for every non-void open set $V\subset L^0(\partial T,E)$ the set $\{n\in\N:\oo_n(f)\in V\}$ has strictly positive lower density. More precisely, we will restrict our attention to more particular subclasses of functions in $E^T$, related to the quantities $w(w,y)\in \mathbb{F}\setminus \{0\}$, $y\in S(x)$, $x\in T$. These functions are called (generalized) harmonic functions on $T$ and will be defined in the next section.
\section{Generalized harmonic functions and universality}\label{sec3}
\noindent

This section is dedicated to the proof of Theorem \ref{one}. Let $w(x,y) \in \mathbb{F}\setminus \{ 0\}$. We start with the following definition.
\begin{Def}\label{def3.1}
A function $f:T\ra E$ is called a generalized harmonic function, if for every $x\in T$ it holds $f(x)=\ssum_{y\in S(x)}w(x,y)f(y)$. The set of such functions is denoted by $H(T,E)$ and is a closed subset of $E^T$ endowed with the cartesian topology induced by the metric
\[
\rho(f,g)=\sum^\infty_{n=0}\frac{1}{2^n}\frac{d(f(x_n),g(x_n))}{1+d(f(x_n),g(x_n))},
\]
where $d$ is the metric in $E$ and $x_n$, $n=0,1,2,\ld$ is a fixed enumeration of $T$.
\end{Def}

We also remind that for $f\in E^T$ the functions $\oo_n(f):\partial T\ra E$ have been defined in Definition \ref{def2.1}.
\begin{Def}\label{def3.2}
Let $\ti\subset\{0,1,2,\ld\}$ be an infinite set. Then a function $f\in H(T,E)$ belongs to the class $U_\ti(T,E)$, if the sequence $\oo_n(f)$, $n\in\ti$ is dense in $L^0(\partial T,E)$ with respect to the topology of convergence in probability induced by the metric
\[
d_\bP(h,g)=\int_{\partial T}\frac{d(h(e),g(e))}{1+d(h(e),g(e))}d\bP(e),
\]
where $d$ is the metric in $E$. Equivalently, $f\in U_\ti(E,T)$ iff, for each $h\in L^0(\partial T,E)$ there exists a sequence $\la_n\in\ti$, $n=1,2,\ld$, such that, $\oo_{\la_n}(f)\ra h$ in probability, as $n\ra +\infty$.
\end{Def}
\begin{rem}\label{rem3.3}
In Definition \ref{def3.2} it is equivalent to require $\la_n<\la_{n+1}$ for all $n=1,2,\ld$\;. If $E=\{0\}$, then $L^0(\partial T,E)$ is a singleton and the result follows. If $E\neq\{0\}$, then, we have already seen that $L^0(\partial T,E)$ does not contain isolated points which implies easily that the sequence $\la_n$, $n=1,2,\ld$ can be chosen strictly increasing. Definition 3.2 is analogous to Menchoff's universality: if a
sequence converges almost everywhere, then it must converge in
probability as well and reversely, if a sequence converges in probability, then it has a 
subsequence that converges almost everywhere.
\end{rem}

\begin{thm}\label{thm3.4}
Let $\ti\in\{0,1,2,\ld\}$ be an infinite set. Then, under the above assumptions and notation the set $U_\ti(T,E)$ is a dense and $G_\de$ subset of $H(T,E)$.
\end{thm}
\noindent
{\bf Proof of Theorem \ref{thm3.4}}. Let $h_j$, $j=1,2,\ld$ be a dense sequence in $L^0(\partial T,E)$, where each $h_j$ is $M_{\ti(j)}$-measurable for some $\ti(j)\in\{0,1,2,\ld\}$. For $n\in\{0,1,2,\ld\}$, $j,s=1,2,\ld$ we consider the set
\[
E(n,j,s)=\bigg\{f\in H(T,E):d_\bP(\oo_n(f),h_j)<\frac{1}{s}\bigg\},
\]
which can be easily seen to be open in $H(T,E)$. One can also check that
\[
U_\ti(T,E)=\bigcap^\infty_{j=1}\bigcap^\infty_{s=1}\bigcup_{n\in\ti}E(n,j,s).
\]
It follows that $U_{\tau} (T,E)$ is $G_{\delta}$ subset of $H(I,E)$. In order to prove that $U_\ti(T,E)$ is dense in $H(T,E)$ we need the following.
\begin{lem}\label{lem3.5}
Let $f:\dis\bigcup^N_{n=0}T_n\ra E$ be a function satisfying
\[
f(x)=\sum_{y\in S(x)}w(x,y)f(y) \ \ \text{for all} \ \  x\in\bigcup^{N-1}_{n=0}T_n.
\]
Let $K>N$. For every $x\in T_N$ let $\bi(x)$ be an element of $T_K$, such that, there exists $e=e_x\in\partial T$ with $\{x,\bi(x)\}\subset e$. Suppose that a function $g:T_K\sm\{\bi(x):x\in T_N\}\ra E$ is given. Then, there exists a function $F:\dis\bigcup^K_{n=0}T_n\ra E$ such that $F \big \vert_{\dis\bigcup^N_{n=0}T_n}=f$, $f(y)=g(y)$ for all $y\in T_K-\{\bi(x):x\in T_N\}$ and
\[
F(z)=\sum_{\el\in S(z)}w(z,\el)F(\el) \ \ \text{for all} \ \ z\in\bigcup^{K-1}_{n=0}T_n.
\]
\end{lem}
\noindent
{\bf Proof of Lemma \ref{lem3.5}} Let $k=K-N\in\{1,2,\ld\}$. We will do an induction on $k$.

For $k=1$, $K=N+1$. Then in the equation $F(x)=\ssum_{y\in S(x)}w(x,y)F(y)$, $x\in T_N$ everything is known except $F(y)$ for $y=\bi(x)$. Since $w(x,\bi(x))\neq0$ in the field $\mathbb{F}$, we can solve this equation for $F(\bi(x))$. Thus, the extension $F$ is obvious and unique.
Suppose we have Lemma \ref{lem3.5} for some value of $k$. We will show it for $k+1$.

Consider the levels $T_N$, $T_{N+k}$, $T_{N+k+1}$. For $x\in T_N$ we have $e_x\in\partial T$ and $\bi(x)\in e_x\cap T_{N+k+1}$ with $x\in e_x$. Let $z=z(x)\in e_x\cap T_{N+k}$. Using the function $g$ and equation (\ref{eq1}) we can compute $\widetilde{g}(y)=F(y)$ for all $y\in T_{N+k}-\{z(x):x\in T_N\}$. By the induction hypothesis, there is a function $F:\dis\bigcup^{N+k}_{n=0}T_n\ra E$ satisfying (\ref{eq1}) for $x\in\dis\bigcup^{N+k-1}_{n=0}T_n$ and extending $f$ and $\tg$.

Now, we know the values $F(z(x))$, for all $x\in T_N$. From the equation $F(z(x))=\ssum_{\el\in S(z(x))}w(z(x),\el)F(\el)$ we can compute $F(\bi(x))$ as well.

Furthermore, for $z\in T_{N+k}-\{z(x):x\in T_N\}$ by construction the value $F(z)=\tg(z)$ is such that
\[
F(z)=\sum_{y\in S(z)}w(z,y)g(y)=\sum_{y\in S(z)}w(x,y)F(y).
\]
The proof of Lemma \ref{lem3.5} is completed. We continue the proof of Theorem \ref{thm3.4}. It remains to show that $U_\ti(T,E)$ is dense in $H(T,E)$. It suffices to consider $\f\in H(T,E)$ and a natural number $N$ and prove that there is $f\in U_\ti(T,E)$, such that
\[
f\Big|_{\bigcup\limits^N_{n=0}T_n}=\f\Big|_{\bigcup\limits^N_{n=0}T_n}.
\]
We remind that we have a sequence $h_j$, $j=0,1,2,\ld$ dense in $L^0(\partial T,E)$, where each $h_j$ is $M_n$-measurable for some $n=n(j)$. Then $h_j$ is $M_m$-measurable for every $m\ge n$.

Since, for every $x\in T$, there exists $y\in S(x)$ with $0<q(x,y)\le\dfrac{1}{2}$, if $n\in N$ is given, we can find $\widetilde{n}>n$, $\widetilde{n}\in\ti$ and such that $h_j$ is $M_{\widetilde{n}}$-measurable and $\bP\Big(\dis\bigcup_{x\in T_n}B(w(x))\Big)<\e$, whatever $\e>0$ and $j$ are given, where $w(x)\in T_{\widetilde{n}}$, $x\in T_n$ are as in Lemma \ref{lem3.5}.

For $\e=1$ we apply this for $n=N$. For $z\in T_{\widetilde{n}}-\{w(x);x\in T_N\}$ we set $F(z)=h_1(z)$ and
\[
F\Big|_{\bigcup\limits^N_{n=0}T_n}=\f\Big|_{\bigcup\limits^N_{n=0}T_n}.
\]
Using Lemma \ref{lem3.5}, the function $F$ has an extension on $\dis\bigcup^{\widetilde{n}}_{n=0}T_n$ satisfying (\ref{eq1}) for $x\in\dis\bigcup^{\widetilde{n}-1}_{n=0}T_n$. Obviously, $\oo_{\tn}(F)$ is $1$ close to $h_1$ with respect to the metric $d_{\bP}$.

After that we consider $F\Big|_{\bigcup\limits^{\widetilde{n}}_{n=0}T_n}$ and we extend it to $\dis\bigcup^m_{n=0}T_n$, so that (\ref{eq1}) is satisfying for $x\in\dis\bigcup^{m-1}_{n=0}T_n$, and the extension is $\e=\dfrac{1}{2}$ close to $h_2$ with respect to $d_\bP$ and $m\in\ti$, $m>\tn$ and $h_2$ is $M_m$-measurable.

Continuing in this way with $h_3,h_4,\ld$ we construct $F\in H(T,E)$ and a sequence $k_n\in\ti$, such that $d_\bP(\oo_{k_n}(F),h_n)<\dfrac{1}{n}$. It follows that $F\in U_\ti(T,E)$ and $F$ coincides with $\f$ on an arbitrary large part $\dis\bigcup^N_{n=0}T_n$ of the tree $T$. Since $\f$ is arbitrary in $H(T,E)$, it follows that $U_\ti(T,E)$ is dense in $H(T,E)$. The proof of Theorem \ref{thm3.4} is\linebreak completed. \qb
\vspace*{0.2cm}

Next we show that we have algebraic genericity.
\begin{thm}\label{thm3.6}
Let $\ti\subset\{0,1,2,\ld\}$ be an infinite set. Under the above assumptions and notation the set $U_\ti(T,E)\cup\{0\}$ contains a vector space dense in $H(T,E)$.
\end{thm}
\noindent
{\bf Proof of Theorem \ref{thm3.6}}. Let $h_j$, $j=1,2,\ld$ be a sequence dense in $H(T,E)$. Since $U_\ti(T,E)$ is dense in $H(T,E)$, there exists $f_1\in U_\ti(T,E)$ so that $\rho(f_1,h_1)<1$. Let $\ti_1\subset\ti$ infinite set such that $\dis\lim_{n\in\ti_1}\oo_n(f_1)=0$ in $L^0(\partial T,E)$. Inductively, we assume that we have infinite sets $\ti\supset\ti_1\supset\cdots\supset\ti_{N-1}$ and functions $f_k\in U_{\ti_{k-1}}(T,E)$ with $\rho(f_k,h_k)<\dfrac{1}{k}$ for $k=1,2,\ld,N$. Let $\ti_N\subset \ti_{N-1}$ be an infinite set, such that $\dis\lim_{n\in\ti_N}\oo_n(f)=0$ in $L^0(\partial T,E)$. Let $f_{N+1}\in U_{\ti_N}(T,E)$ such that $\rho(f_{N+1},h_{N+1})<\dfrac{1}{N+1}$.

Therefore, we have defined a sequence $f_j$, $j=1,2,\ld$. We will show that the linear span $A=\lan f_j,\; j=1,2,\ld\ran$ is dense in $H(T,E)$ and that $A\subset U_\ti(T,E)\cup\{0\}$. Since the sequence $h_k$ is dense in $H(T,E)$, which does not contain isolated points and $\dis\lim_k\rho(f_k,h_k)=0$ it follows that the sequence $\Big\{f_k\Big\}^\infty_{k=1}\subset A$ is dense in $H(T,E)$.

Let $L=a_1f_1+\cdots+a_mf_m$ with $a_k\in \mathbb{F}$, $a_m\neq0$, and let $h\in L^0(\partial T,E)$. Since $f_m\in U_{\ti_{m-1}}(T,E)$, there exists a subsequence $\la_n$, $n=1,2,\ld$ in $\ti_{m-1}\subset\ti_{m-2}\subset\cdots\subset\ti_1\subset\ti$ so that
\[
\lim_n\oo_{\la_n}(f_m)=a^{-1}_m h.
\]
Since $\la_n$ is a subsequence of $\ti_{m-2}\subset\cdots\subset\ti_1\subset\ti$ we have $\dis\lim_n\oo_{\la_n}(f_k)=0$ for $k=1,\ld,m-1$. It follows that  $\dis\lim_n\oo_{\la_n}(L)=h$.

Thus, $A-\{0\}\subset U_\ti(T,E)$ and the proof is completed. \qb
\section{Frequent Universality}\label{sec4}
\noindent
We turn now to frequent universality and this section is dedicated to the proof of Theorem \ref{two}.
\begin{Def}\label{def4.1}
A generalized harmonic function $f\in H(T,E)$ is frequently universal, if, for every non-empty open set $V\subset L^0(\partial T,E)$ the set $\{n\in\{0,1,2,\ld\}:\oo_n(f)\in V\}$ has strictly positive lower density. The set of such functions is denoted by $U_{FM}(T,E)$.
\end{Def}
\begin{thm}\label{thm4.2}
Under the above assumptions and notation the class $U_{FM}(T,E)$ is dense in $H(T,E)$.
\end{thm}
\noindent
{\bf Proof of Theorem \ref{thm4.2}}. First we show that $U_{FM}(T,E)\neq\emptyset$. If $g\in U_\ti(T,E)$ with $\ti=\{0,1,2,\ld\}$, then $h_k=\oo_k(g)$, $k=0,1,2,\ld$ is a dense in $L^0(\partial T,E)$ sequence and each $h_k$ is $M_k$-measurable. Fix such a sequence $h_k$, $k=0,1,2,\ld$\;.

In order to construct a function $f\in U_{FM}(T,E)$ it suffices that for every $k\in\{0,1,2,\ld\}$ the set $\Big\{n\in\{0,1,2,\ld\}:\oo_n(f)\in B{(h_k,\frac{1}{2^k})}\Big\}$ has strictly positive lower density, where
\[
B{\left(h_k,\frac{1}{2^k}\right)}=\bigg\{h\in L^0(\partial T,E):d_\bP(h,h_k)<\frac{1}{2^k}\bigg\}.
\]
We set $\el(k)=s+1$ where $k=2^s\cdot d$, $d$ odd and $r_k=\ssum^k_{n=1}\el(n)$. Lemma 3.4 of \cite{7} gives that
\begin{enumerate}
\item[i)] $r_{2^n}=2^{n+1}-1$
\item[ii)] For $m\le n$ it holds card$(\{k:1\le k\le 2^n:\el(k)=m\})=2^{n-m}$  and
\item[iii)] For $m\ge1$ the set $\{r_n:n\in\{1,2,\ld\}:\el(n)=m$ has strictly positive lower density.
\end{enumerate}

Thus, it suffices to have $\oo_{r_k}(f)\in B{\Big(h_{\el(k)},\frac{1}{2^{\el(k)}}\Big)}$.\vspace*{0.2cm}
The proof of the following lemma is similar to the proof of Lemma \ref{lem3.5} and is omitted.
\begin{lem}\label{lem4.3}
Let $n>0$ and $s>0$ be two natural numbers and $\f:T_0\cup\cdots\cup T_s\ra E$ be a function satisfying (\ref{eq1}) for every $x\in T_0\cup\cdots\cup T_{s-1}$. Then there exists an extension $\vPsi$ of $\f$, $\vPsi:T_0\cup\cdots\cup T_{s+n}\ra E$ satisfying (\ref{eq1}) for every $x\in T_0\cup\cdots\cup T_{s+n-1}$, such that, $\oo_{s+n}(\vPsi)\in B{(h_n,\frac{1}{2^n})}$.
\end{lem}

Using Lemma \ref{lem4.3} we construct a function $f\in U_{FM}(T,E)$. Indeed, we apply Lemma \ref{lem4.3} for $k=1,2,3\ld$ for $s=r_{k-1}$, $n=\el(k)$, $s+n=r_k$. Then, we obtain $\oo_{r_k}(f)\in B{(h_{\el(k)},\frac{1}{2^{\el(k)}})}$. Thus, the lower density of the set $\{r_n:n\in\{0,1,\ld\},\el(n)=m\}$ is strictly positive. Therefore, $f\in U_{FM}(T,E)$ which implies that $U_{FM}(T,E)$ is non-void.

In order to prove that the set $U_{FM}(T,E)$ is dense in $H(T,E)$, it suffices to observe that in the previous construction of $f$ in $U_{FM}(T,E)$ we can arrange that an initial part of $f$ can be the restriction of any $\f\in H(T,E)$ on $T_0\cup\cdots\cup T_N$ for any given natural number $N$ and then apply Lemma \ref{lem4.3} for $s=r_{k-1}\ge N$, $n=\el(k)$, $s+n=r_k$; that is for $k\ge k(N)$. \qb
\begin{rem}\label{rem4.4}
Let $\ti=\{0=n_0<n_1<n_2<\ld\}$ be an infinite subset of $\{0,1,2,\ld\}$. Then a function $f\in H(T,E)$ belongs to $U^\ti_{FM}(T,E)$ if, for every non-empty open set $V\subset L^0(\partial T,E)$ the  set $\{k:\oo_{n_k}(f)\in V\}$ has strictly positive lower density. Then we can show that $U^\ti_{FM}(T,E)$ is dense in $H(T,E)$. One way to do this is to consider a new tree $\widetilde{T}=T_0\cup T_{n_1}\cup T_{n_2}\cup \ld$\;. For $x\in\tT$ the set of children in $\tT$ of $x$ denoted by $\tS(x)$ is defined as follow.

For $x\in T_{n_k}$, $k\in\{0,1,2,\ld\}$ the set $\tS(x)$ contains exactly all $y\in T_{n_{k+1}}$, such that, there exists $e\in\partial T$ with $\{x,y\}\subset e$. If $e\supset\{x=x_1,x_2,\ld,x_{m-1},x_m=y\}$  with $x_j\in S(x_{j-1})$, $j=2,\ld,m$, then we set
\[
\tq(x,y)=q(x_1,x_2)\cdot q(x_2,x_3)\cdot\,\cdots\,\cdot q(x_{m-1},x_m)>0 \ \ \text{and}
\]
\[
\widetilde{w}(x,y)=w(x_1,x_2)\cdot w(x_2,x_3)\cdot\,\cdots\,\cdot w(x_{m-1},x_m)\in F-\{0\}.
\]

It can be seen that if $f\in H(T,E)$, then, $f|_{\tT}\in H(\tT,E)$ and that the map $H(T,E)\ni f\ra f|_{\tT}\in H(\tT,E)$ is an onto isomorphism. Also $f\in U^\ti_{FM}(T,E)$ if and only if $f|_{\tT}\in U_{FM}(\tT,E)$.

According to Theorem \ref{thm4.2} the set $U_{FM}(\tT,E)$ is dense in $H(\tT,E)$. Thus, the inverse image of $U_{FM}(\tT,E)$, which is $U^\ti_{FM}(T,E)$ is dense in $H(T,E)$.

\end{rem}
\section{An additional assumption}\label{sec5}
\noindent

In this section, we prove Theorem \ref{three}. We assume in addition that for all $x\in T$ it holds
\begin{eqnarray}
\sum_{y\in S(x)}w(x,y)=1\in F.  \label{eq2}
\end{eqnarray}
Under this assumption, if $f:\dis\bigcup^N_{n=0}T_n\ra E$ is a function satisfying (\ref{eq1}) for all $x\in\dis\bigcup^{N-1}_{n=0}T_n$, then we can extend it on $\dis\bigcup^{N'}_{n=0}T_n$, $N<N'$ (or even on $T=\dis\bigcup^\infty_{n=0}T_n\Big)$, so that the extension satisfies $f(y)=f(x)$ when $y\in S(x)$ and $x\in\dis\bigcup^{N'-1}_{n=N}T_n$ (or $x\in\dis\bigcup^\infty_{n=N}T_n$, respectively). In the latter case the extension belongs to $H(T,E)$. Moreover, the set of such functions is dense in $H(T,E)$.

Under assumption (\ref{eq2}) and the previous assumptions, we will show that $U_{FM}(T,E)\cup\{0\}$ contains a vector space dense in $H(T,E)$.

Towards this end we consider the space $E^{\N}$, which is a topological vector space over the field $F$ under the cartesian topology. It has a translation invariant metric $\widetilde{d}$
\[
\widetilde{d}(x,y)=\sum^\infty_{n=0}\frac{1}{2^n}\frac{d(x_n,y_n)}{1+d(x_n,y_n)}
\]
where $x=(x_1,x_2,\ld)$, $y=(y_1,y_2,\ld)$ $x_j,y_j\in E$, for all $j=0,1,2,\ld$\;.

From the results of the previous sections the set $U_{FM}(T,E^\N)$ is dense in $H(T,E^\N)$.
\begin{clm}\label{clm5.1}
If $f=(f_0,f_1,\ld)\in U_{FM}(T,E^\N)$ where $f_j\in H(T,E)$, then the linear span $\lan f_0,f_1,\ld\ran$ is contained in $U_{FM}(T,E)\cup\{0\}$.
\end{clm}

The proof is similar to the proof of Claim 4.1 in \cite{7} and is omitted. The only difference is that since $a_k\in F$ and not to $\C$ the absolute values $|a_k|$ make no sense and we multiply some balls with $a^{-1}_k$ and not with $|a_k|^{-1}$.
\begin{clm}\label{clm5.2}
Under the assumption (\ref{eq2}) and all other assumptions, there exists a sequence $f_0,f_1,\ld,f_n,\ld$ dense in $H(T,E)$ such that $f=(f_0,f_1,f_2,\ld)$ belongs to $U_{FM}(T,E^\N)$.
\end{clm}

The proof is similar to the proof of Claim 4.2 in \cite{7} and is omitted.

Combining Claim \ref{clm5.1} with Claim \ref{clm5.2} we deduce the following.
\begin{thm}\label{thm5.3}
Under assumption (\ref{eq2}) and the other assumptions the set $U_{FM}(T,E)\cup\{0\}$ contains a vector space dense in $H(T,E)$.
\end{thm}

Combining this with Remark \ref{rem4.4} we obtain the following.
\begin{thm}\label{thm5.4}
Under assumption (\ref{eq2}) and the other assumptions and notation, the set $U^\ti_{FM}(T,E)\cup\{0\}$ contains a vector space dense in $H(T,E)$, for every infinite set $\ti\subset\{0,1,2,\ld\}$.
\end{thm}

Next we will show that under assumption (\ref{eq2}) the set $U_{FM}(T,E)$ is disjoint from a $G_\de$ and dense subset of $H(T,E)$. In case where $E$ is a complete metric space we can say that $U_{FM}(T,E)$ is meager (and dense) in $H(T,E)$. We have similar results for the set $U^\ti_{FM}(T,E)$, for any infinite set $\ti\subset\{0,1,2,\ld\}$.
\begin{Def}\label{def5.5}
$X(T,E)$ is the set of functions $f\in H(T,E)$, such that, for every non-empty open set $V\subset L^0(\partial T,E)$, the upper density of the set $\{n\in\N:\oo_n(f)\in V\}$ is equal to 1.
\end{Def}
\begin{clm}\label{clm5.6}
$X(T,E)\cap U_{FM}(T,E)=\emptyset$.
\end{clm}

The proof of Claim \ref{clm5.6} is similar to the proof of Proposition 3.9 in \cite{7} and is omitted.
\begin{clm}\label{clm5.7}
The set $X(T,E)$ is a $G_\de$ subset of $H(T,E)$.
\end{clm}

The proof of Claim \ref{clm5.7} is similar to part of the proof of Proposition 3.8 in \cite{7} and is omitted.
\begin{prop}\label{prop5.8}
Under assumption (\ref{eq2}) and the other assumptions and notation the set $X(T,E)$ is dense in $H(T,E)$.
\end{prop}
\noindent
{\bf Sketch of the proof}

Let $\f\in H(T,E)$ and $N\in\{0,1,2,\ld\}$. It suffices to construct $f\in X(T,E)$ so that
\[
f\Big|_{\bigcup\limits^N_{n=0}T_n}=\f\Big|_{\bigcup\limits^N_{n=0}T_n}.
\]

Since $E$ is separable, there exists a sequence $B_n$ of balls in $L^0(\partial T,E)$, which is a base for the topology of $L^0(\partial T,E)$.

We extend $\f\Big|_{\bigcup\limits^N_{n=0}T_n}$ to a function $\f_1:\dis\bigcup^{N_1}_{n=0}T_n\ra E$ where $N_1>N$ so that (\ref{eq1}) holds for all $x\in\dis\bigcup^{N_1-1}_{n=0}T_n$ and $\oo_{N_1}(f)\in B_1$.  Then, because of assumption (\ref{eq2}), we can further extend $\f_1$ so that $\oo_{N_1}(\f_1)=\oo_{N_1+1}(\f_1)=\cdots=\oo_{N_1+k_1}(\f_1)\in B_1$, where $k_1$ is big enough so that $\dfrac{k_1}{N_1+k_1}$ is very close to 1.

Then we extend again this function so that $\oo_n(\f_1)\in B_1$ for a lot of $n$'s and then we do the same so that $\oo_n(\f)\in B_2$ for a great numbers of consecutive $n$'s.

After that we repeat the same with $B_1,B_2,B_3$ and so on.

In this way we define $f\in H(T,E)$ so that for every $n \in\{1,2,3,\ld\}$ the upper density of the set $\{k\in\N:\oo_k(f)\in B_n\}$ is equal to 1. This function $f$ belongs to $X(T,E)$. Because the initial part of $f$ coincides with the initial part of $\f\in H(T,E)$ and this is possible for every $\f$, the set $X(T,E)$ is dense in $H(T,E)$. \qb\vspace*{0.2cm}

Combining the above facts we deduce.
\begin{thm}\label{thm5.9}
Under assumption (\ref{eq2}) and the other assumptions $U_{FM}(T,E)$ (and $U^\ti_{FM}(T,E)$) is disjoint from a $G_\de$ and dense subset of $H(T,E)$.

\end{thm}

\noindent
N. Biehler \\
National and Kapodistrian University of Athens \\
Department of Mathematics,\\
%
%
%
e-mail: nikiforosbiehler@outlook.com \vspace*{0.3cm}\\
\noindent
V. Nestoridis \\
National and Kapodistrian University of Athens \\
Department of Mathematics,\\
%
%
%
e-mail: vnesto@math.uoa.gr\vspace*{0.3cm} \\
E. Nestoridi\\
Princeton University \\
Department of Mathematics,\\
e-mail: exn@princeton.edu

\end{document}